\documentclass[11pt]{amsart}
\usepackage{graphicx}
\usepackage{amsaddr}
\usepackage{hyperref}
\usepackage{amsmath,amssymb}
\usepackage{dirtytalk}
\large\normalsize

\providecommand{\ts}{\hspace{15pt}}

\newcommand{\mb}{\mathbb}
\newcommand{\mc}{\mathcal}

\theoremstyle{plain}
\newtheorem{Th}{Theorem}[section]
\newtheorem{Lemma}[Th]{Lemma}
\newtheorem{Cor}[Th]{Corollary}

\theoremstyle{plain}

\numberwithin{equation}{section}
\begin{document}
	\newcommand{\T}{\mathbb{T}}
	\newcommand{\R}{\mathbb{R}}
	\newcommand{\Q}{\mathbb{Q}}
	\newcommand{\N}{\mathbb{N}}
	\newcommand{\Z}{\mathbb{Z}}
	\newcommand{\tx}[1]{\quad\mbox{#1}\quad}
	\noindent {\tt The Nepali Math. Sc. Report\\
		Vol. 36, No.1, 2019\\}

\title[Maximum Norm Estimates  ]{A priori estimates in terms of the maximum norm for the solution of the Navier-Stokes equations with periodic initial data}


\author[Santosh Pathak]{ \bfseries Santosh Pathak}
\address{ Department of Mathematics and Statistics\\
 University of New Mexico, Albuquerque, NM 87131, USA \\
 spathak@unm.edu}

\vspace{-.2cm} 
\thanks{\hspace{-.5cm}\tt
	Received  May 31, 2016
	\hfill }

\maketitle
\thispagestyle{empty}

{\footnotesize
\noindent{\bf Abstract:}  In this paper, we consider the Cauchy problem for the incompressible Navier-Stokes equations in $\mathbb{R}^n$ for $n\geq 3 $  with smooth periodic initial data and derive a priori estimtes of the maximum norm of all derivatives of the solution in terms of the maximum norm of the initial data. This paper is a special case of a paper by H-O Kreiss and J. Lorenz which also generalizes the main result of their paper to higher dimension.
\bf AMS (MOS) Subject Classification.} Classification here.

\section{Introduction}
 We consider the Cauchy problem of the Navier-Stokes equations in $\mb R^n, n \geq 3$: 
\begin{align}
 u_t + u \cdot \nabla u + \nabla p = \triangle u, \ts \nabla \cdot u = 0, 
\end{align}
with initial condition 
\begin{align}
u(x,0)= f(x), \ts x \in  \mb R^n, 
\end{align}
where $u=u(x,t)= (u_1(x,t), \cdots u_n(x,t)) $ and $ p=p(x,t) $ stand for the unknown velocity vector field of the fluid and its pressure, while $f=f(x)=(f_1(x),\cdots  f_n(x) ) $ is the given initial velocity vector field. In what follows, we will use the same notations for the space of vector valued and scalar functions for convenience in writing.
\medskip

There is a large literature on the existence and uniqueness of solution of the Navier-Stokes equations in $\mb R^n $.  For given initial data, solutions of (1.1) and (1.2)  have been constructed in various function spaces. For example, if $f \in L^r $ for some $r$ with $ 3 \leq r < \infty $, then it is well known that there is a unique classical solution in some maximum interval of time $0\leq t < T_f $ where $ 0 < T_f \leq \infty $. But for the uniqueness of the pressure one requires $| p(x,t)| \to 0$ as $|x| \to \infty$. (See \cite{Kato} and \cite{Wiegner}  for $r=3 $ and \cite{Amann}  for $ 3 < r < \infty $.)
\medskip 

If $f\in L ^{\infty}(\mb R^n) $ then existence of a regular solution follows from \cite{Cannon}. The solution is only unique if one puts some growth restrictions on the pressure as $| x| \to \infty$. A simple example of non-uniqueness is demonstrated in \cite{Kim} where the velocity $u$ is bounded  but $|p(x,t)| \leq C |x| $. In addition, an estimate $|p(x,t)|\leq C(1+|x|^{\sigma} ) $ with $\sigma <1$ ( see \cite{Galdi} ) implies uniqueness. Also the assumption $ p \in L^1_{loc}(0,T; BMO) $ (see \cite{Giga}) implies uniqueness. 
\medskip 

In this paper we consider the initial function $f \in C^{\infty}_{per}(\mb R^n) $ which is the space of smooth $2 \pi $  periodic functions.  Since $  C^{\infty}_{per}(\mb R^n) $ is a closed subspace of the Banach space $ L^{\infty}(\mb R^n) $, the existence of a regular solution of the Navier-Stokes equations (1.1) and (1.2) can be guaranteed by \cite{Cannon}. In addition, in a paper by Giga and others \cite{Giga} they consider $f \in BUC(\mb R^n)$ where  $ BUC(\mb R^n)  $ is the space of bounded uniformly continuous functions. In the paper they construct a regular solution of the Navier-Stokes equations in some maximum interval of time $ 0 \leq t < T_f $ where $ T_f \leq \infty $. Clearly, our case is also a special case of their paper where we put extra assumption of \say{smooth periodic} on their initial function $f \in BUC(\mb R^n)$. Moreover, for smooth periodic initial data, the existence of smooth periodic solution is proved  by H-O Kreiss and J. Lorenz in their book \cite{Kreiss} for $n=3$ where they use successive iteration using the the vorticity formulation. On the other hand, Giga and others use iteration on the integral equation of the transformed abstract ordinary differential equations to construct a mild solution of the Navier-Stokes equations and later prove such mild solution is indeed a regular solution (local in time)  of the Navier-Stokes equations (1.1) and (1.2) for $f \in BUC(\mb R^n)  $ . Readers are referred to the paper by Giga and others  \cite{Giga} for details on existence of the smooth periodic solution of the Navier-Stokes equations of (1.1) and (1.2) for $ f \in C^{\infty}_{per}(\mb R^n)$ with necessary alternations in their proofs of case $ f \in BUC(\mb R^n)$.
\medskip 

The work in this paper reproves  Theorem 4.1 of the Kreiss and Lorenz paper \cite{Lorenz}  in  periodic case assuming smooth periodic solution exists for some maximum interval of time $ 0 \leq t < T_f$. Since we are in a special case of their paper, result of Theorem 4.1 must be true for smooth periodic initial data  as well, but what makes our work interesting and different is the approach taken to handle the pressure term of the Navier-Stokes equations while deriving the result of Theorem 4.1 of the Kreiss and Lorenz paper as I have adopted in my first paper \cite{SP}   Notice, pressure term of the Navier-Stokes equations  can be determined from the Poisson equation 
\begin{align}
\triangle p = - \nabla \cdot ( u \cdot \nabla ) u 
\end{align}
which is given by 
\begin{align}
p = \sum_{i,j} R_i R_j ( u_i u_j) ,
\end{align}
where $R_i = (-\triangle )^{-1/2} D_i  $ is the $i$-th Riesz transform. Since the Riesz transforms are not bounded in $L^{\infty}(\mb R^n)$, the pressure term $p \in L^1_{loc} (0,T;BMO) $ where $ BMO$ is the space of functions of bounded mean oscillation. Because of the non-local nature of the pressure, the proof of Theorem 4.1 of the Kreiss and Lorenz paper  is complicated, however.
\medskip 

The main objective of this paper is to derive a priori estimates of the maximum norm of the derivatives of $u$ in terms of the maximum norm of the initial function, $u(x,0)=f(x)$, {\it assuming} the solution to exist and to be $C^{\infty}_{per}(\mb R^n), n \geq 3 $ for $ 0 \leq t < T_f$. Before we start formulating the problem, we introduce the following notations 

\begin{align*}
| f|_{\infty} = \mathop{\sup}_{x} { |f(x)|}  \ts \text{with} \ts |f(x)|^2= \sum_{i} f_i^2(x), 
\end{align*}
and $D^{\alpha} =D_1^{\alpha_1} \cdots D_n^{\alpha_n},  D_i = \partial /{\partial x_i } \hspace{5pt} \text{for a multiindex} \hspace{5pt} \alpha =(\alpha_1,\cdots, \alpha_n )$.
In what follows, for any $j=0,1,\cdots $, if  $ | \alpha | = j $ then we will denote $D^{\alpha} $ by $D^j$. We also set 
\begin{align*}
|\mathcal {D}^j u(t) |_{\infty} :=|\mathcal {D}^j u(\cdot, t ) |_{\infty} =\mathop{\max}_{|\alpha|=j}  |D^{\alpha} u(\cdot,t)|_{\infty} .
\end{align*}
Clearly, $|\mathcal{D}^j u(t) |_{\infty} $ measures all space derivatives of order $j$ in maximum norm.
\medskip 

Proving the following  theorem is the main goal of this paper whereas Kreiss and Lorenz in their paper \cite{Lorenz} prove the same theorem for $ f \in L^{\infty}(\mb R^n) $ for $n=3$ from rather difficult approach while dealing with the pressure term $p(x,t)$.
\bigskip 

\begin{Th}
Consider the Cauchy problem for the Navier-Stokes equations (1.1), (1.2), where $ f \in C^{\infty}_{per}  ( \mb R^n)$  for $n \geq 3$ with  $  \nabla  \cdot f = 0 $. There is a constant $c_0>0$ and for every $j=0,1,\cdots $ there is a constant $K_j$ so that 
\begin{align}
t^{j/2} | \mc D^j u(t) |_{\infty} \leq K_j | f|_{\infty} \ts \text{for} \ts 0 < t \leq  \frac{c_0}{|f|_{\infty}^2}.
\end{align}
The constants $c_0$ and $K_j$ are independent of $t$ and $f$.
\end{Th}
\bigskip

For the purpose of proving Theorem 1.1, we start by transforming the momentum equation (1.1) of  the Navier-Stokes equations  into the abstract ordinary differential equation for $u$
\begin{align}
u_t = \triangle u - \mb P (u \cdot \nabla ) u 
\end{align}
by eliminating the pressure, where $\mb P $ is the Leray projector  defined by 
\begin{align*}
\mb P = (  \mb P_{ij})_{1\leq i,j \leq n } , \ts  \mb P_{ij} = \delta_{ij} + R_i R_j; 
\end{align*}
where $ R_i$ is same as  in (1.4) and $ \delta_{ij}$ is the Kronecker delta function. Note that the equation (1.6) is obtained from (1.1 ) by applying the Leray projector with the properties $\mb P( \nabla p)=0, \mb P ( \triangle u ) = \triangle u$, since $\nabla \cdot u = 0 $. 
\medskip 

Since $ \mb P (u \cdot \nabla u)= \sum_{i} D_i \mb P(u_i u)$, therefore it is very appropriate to consider an analogous  system of (1.6) as below:
\begin{align}
u_t = \triangle u + D_i \mb P g(u)  \ts x \in \mb R^n, \ts t>0
\end{align}
with initial condition
\begin{align}
u(x,0)= f(x) \ts \text{where} \ts f \in C^{\infty}_{per}(\mb R^n).
\end{align}
Here $ g: \mb R^n \to \mb R^n $ is assumed to be quadratic in $u$. The maximal interval of existence is again $ 0 \leq t < T_f$. We would like to prove the estimates of the maximum norm of the derivatives of the solution of (1.7) and (1.8)  in terms of the maximum norm of the initial data.

\bigskip

\begin{Th} \label{main}  Under the above assumptions on $f$ and $g$ the solution of (1.7) and (1.8) satisfies the following  \\
(a) There is a constant $c_0 >0$ with 
\begin{align}
T(f) >  \frac{c_0}{|f|_{\infty}^2} 
\end{align}
and  
\begin{align}
|u(t)|_{\infty} \leq 2 |f|_{\infty} \ts \text{for} \ts  0\leq t \leq \frac{c_0}{|f|_{\infty}^2} . 
\end{align}
(b) For every $j=1,2,\cdots $, there is a constant $K_j>0$ with 
\begin{align}
t^{j/2} | \mc{D}^j u(.,t)|_{\infty} \leq K_j |f|_{\infty}  \ts \text{for}  \ts 0< t  \leq  \frac{c_0}{|f|_{\infty}^2}.
\end{align}
The constant $c_0$ and $K_j$ are independent of $t$ and $f$.
\end{Th}
\medskip 

In section 2, we will introduce some auxiliary results for the solution of the heat equation and few other important estimates which are used later in section 3 and 4. Proof of Theorem 1.2 will be provided in section 3. Then we prove Theorem 1.1 in  section 4. Finally, in section 5  we outline some remarks on the use of the result  obtained in Theorem 1.1.
\bigskip

\section{Some Auxiliary results} 
\bigskip
Let us consider $f\in C^{\infty}_{per} ( \mb R^n) $. The solution of 
\begin{align*}
u_t = \triangle u , \ts u = f \ts \text{at} \ts t=0,
\end{align*}
is denoted by 
\begin{align*}
u(t):=u(\cdot, t ) = e^{ \triangle t } f =  \frac{1}{(2\pi)^n} \int_{\mb T^n} \theta (x-y,t) f(y) dy 
\end{align*}
where 
\begin{align}
\theta(x,t)=\sum_{ k \in \mathbb{Z}^n} e^{-|k|^2 t } e^{ik \cdot x}  , \ts t>0
\end{align}
is the periodic heat kernel in $\mb R^n$.
Using the Poisson summation formula, (2.1) can be written as 
\begin{align}
\theta(x,t)= \sum_{k\in \mathbb{Z}^n} \bigg(\frac{\pi}{t}\bigg)^n \exp\bigg[\frac{-|x+2\pi k|^2}{4t} \bigg], \ts t > 0. 
\end{align}
With the use of (2.2), it is well known that 
\begin{align}
|e^{t \triangle } f |_{\infty} \leq |f|_{\infty}, \ts t\geq 0 
\end{align}
and 
\begin{align}
| \mc D^j e^{t \triangle} f |_{\infty} \leq C_j t^{-j/2} | f |_{\infty} 
\end{align}
for some $C_j>0$ independent of $t$ and $f$.

\bigskip

\begin{Lemma} 
 Let $ f \in C^{\infty}_{per}( \mb R^n ) $ then  for any $j \geq 1 $ 
\begin{align}
| \mc D^j e^{\triangle t } \mb P f  |_{\infty} \leq C_j t^{-j/2} | f |_{\infty} \ts \text{for} \ts t>0 
 \end{align}
  for some constant $C_j>0$ independent of $t$ and $f$.
\end{Lemma}
\begin{proof}
Let us first denote $e^{\triangle t} f = \theta * f  $ where $\theta(x,t) $ is given by (2.2). Let us denote the Fourier coefficient of a function by $\mc F$. For  $ \xi \in \mb Z ^n$, notice $ \mc F ( {\theta (x,t)}) (\xi ) = C e^{-t |\xi|^2 },   t > 0 $, where $C$ depends on the  normalizing constant in the definition of the Fourier coefficient. In the proof of this lemma, we will allow the constant $C$ to change line to line as per the need.  Now, for any $t > 0$, any choice of $k,l \in \{1,2, \cdots , n  \}$; and for any multiindex $ \alpha $ such that $ |\alpha|=j$, the operator $D^j e^{\triangle t}  \mb P_{k l}  $ on the Fourier side is given by 
\begin{align*}
\mc F ( D^j e^{ \triangle t}  \mb P_{kl} f_l)(\xi)  & =(-i \xi)^{\alpha} \mc{F}(e^{ \triangle t} \mb P_{kl} f_l )(\xi) \\
&= (-i \xi)^{\alpha}  \mc F( \theta * \mb P_{kl} f_l ) (\xi) \\ 
& = C (-i \xi)^{\alpha}  \mc{F}(\theta(x,t))(\xi) \mc{F}(\mb P_{kl} f_l )(\xi) \\
&=C (-i \xi)^{\alpha} e^{-t |\xi|^2 } \bigg( \delta_{kl} - \frac{\xi_k \xi_l}{| \xi |^2} \bigg) \mc F (f_l) (\xi) \\
& = C (-i \xi)^{\alpha} e^{-t |\xi |^2 } \delta_{k l} \mc F (f_l) (\xi)\\
& -C  (-i \xi)^{\alpha}
 \xi_k \xi_l \mc F (f_l) (\xi)  \int_t^{\infty} e^{-\tau | \xi |^2 } d\tau. 
\end{align*}
Using Fourier expansion we can write 
\begin{align*}
D^j( e^{ \triangle t} \mb P_{kl} f_l)(x) & =C \sum_{\xi \in \mb Z^n}   (-i\xi)^{\alpha} \delta_{kl} e^{-t |\xi|^2} \mc F (f_l)(\xi) e^{i \xi \cdot x } \\
&+ C \sum_{\xi \in \mb Z^n}  (-i\xi)^{\alpha} (i\xi_k) (i\xi_l) \mc F(f_l) (\xi) e^{i \xi \cdot x } \int_t^{\infty}  e^{-\tau  |\xi|^2 } d\tau  \\
& = (-1)^j C  \delta_{kl} D^{\alpha} \sum_{\xi \in \mb Z^n}   e^{-t |\xi|^2} \mc F(f_l) (\xi) e^{i \xi \cdot x }  \\
&+ C (-1)^j \int_t^{\infty} \sum_{\xi \in \mb Z^n} e^{-\tau  | \xi|^2 } (i \xi)^{\alpha}  (i\xi_k) (i\xi_l)  \mc F (f_l) (\xi) e^{i \xi \cdot x } d\tau \\
&= (-1)^j C \delta_{kl}  D^{\alpha} e^{ \triangle t} f_l + (-1)^j  C  \int_t^{\infty} D^{\alpha} D_k D_l e^{ \triangle \tau} f_l d\tau \\
&=I_1 + I_2.
\end{align*}
From (2.4) we have  $|I_1|_{\infty} \leq C_j t^{-j/2} |f_l |_{\infty} $.
\medskip
By the use of (2.4) one more time we obtain 
\begin{align*}
|I_2 |_{\infty}  &  \leq  C_j   |f_l|_{\infty} \int_t^{\infty} \tau^{-(j+2)/2} d\tau \\
& \leq C_j  t^{-j/2} | f_l |_{\infty} .
\end{align*}
Therefore 
\begin{align*}
| D^j e^{ \triangle t} \mb P_{kl} f_l  |_{\infty} &  \leq | I_1|_{\infty} + | I_2|_{\infty} \\
& \leq C_j t^{-j/2} |f_l |_{\infty}.
\end{align*}
Hence Lemma 2.1 is proved.
 \end{proof}

 \medskip 
 \begin{Cor}
 Let $ g\in C^{\infty}_{per} ( \mb R^n  \times [0,T] ) $ for some $T>0$, then the solution of
 \medskip
 \begin{align}
 u_t = \triangle u + D_i \mb P g , \ts u = 0 \ts \text{at} \ts t=0  
 \end{align}
 satisfies 
 \begin{align}
  |u(t)|_{\infty} \leq C t^{1/2} \mathop{\max}_{0\leq s \leq t } |g(s) |_{\infty}. 
 \end{align}
 
\end{Cor}

 \begin{proof}
The solution of (2.6) is given by 
 \begin{align*}
 u(t) = \int_0^t e^{\triangle (t-s) } D_i  \mb P g(u))(s)  ds 
 \end{align*}
 and 
 \begin{align*}
 |u(t)|_{\infty} \leq \int_0^t | e^{\triangle (t-s) }  D_i \mb P g(u)(s) |_{\infty} ds. 
 \end{align*}
After commuting  $D_i $ with the heat semi-group, we can use Lemma 2.1 to obtain
 \begin{align*}
 |u(t)|_{\infty} \leq \mathop{\max}_{0\leq s \leq t } |g(s)|_{\infty} \int_0^t (t-s)^{-1/2} ds .
 \end{align*}
 Hence we obtain 
 \begin{align*}
 |u(t)|_{\infty} \leq C t^{1/2} \mathop{\max}_{0\leq s \leq t } |g(s) |_{\infty} .
 \end{align*}
  \end{proof}

 \section{Estimates for  $u_t = \triangle u + D_i \mb P g(u) $ : proof of Theorem 1.2}
  \bigskip
In this section we consider the system  $u_t = \triangle u + D_i \mb P g(u) $ with the initial condition $u=f $ at $ t=0$ where $ f\in C^{\infty}_{per}(\mb R^n) $. It is well-known that the solution is smooth $2\pi$ periodic in a maximal interval $ 0 \leq t < T_f$ where $ 0 < T_f \leq \infty $.
 
 Let us consider $u$ is the solution of the inhomogeneous equation
 $ u_t = \triangle u + D_i \mb P (g(u(x,t))) $  and recall $ g(u)$ is quadratic in $u$. Thus, there is a constant $ C_g$ such that  we have the following:
 \begin{align}
 | g(u) | \leq C_g |u|^2, \ts |g_u(u) | \leq C_g |u|, \ts \text{for all} \ts u \in \mb R^n 
 \end{align}
 We first estimate the maximum norm of $u$.
 \bigskip
 
 \begin{Lemma}
 Let $C_g$ denote the constant in (3.1) and let $C$ denote the constant in (2.7); set $c_0 =\frac{1}{16C^2 C_g^2} $. Then we have $T_f > c_0/{ | f |_{\infty}^2 } $ and 
 \begin{align}
 | u (t) |_{\infty} < 2 | f |_{\infty} \ts \text{ for } \ts 0 \leq t < \frac{c_0}{|f|_{\infty}^2 }.
 \end{align}
 \end{Lemma}
 \begin{proof}
 Suppose (3.2) does not hold, then we can find the smallest time $t_0$ such that $ |u(t_0) |_{\infty} = 2 | f |_{\infty} $. Since $t_0$ is the smallest time so we have $ t_0 < c_0/|f|_{\infty}^2 $. Now by (2.3) and (2.7) we have 
 \begin{align*}
 2|f|_{\infty}  & = |u(t_0) |_{\infty} \\
 & \leq |f|_{\infty} + C t_0^{1/2} \mathop{\max}_{0\leq s \leq t_0 } | g(s)  |_{\infty} \\
 & \leq |f|_{\infty} + C C_g t_0^{1/2} \mathop{\max}_{0 \leq s \leq t_0}  | u (s) |_{\infty}^2  \\
 & \leq | f|_{\infty} + C C_g t_0^{1/2} 4 |f |_{\infty}^2 .
 \end{align*}
 This gives 
 \begin{align*}
 1 \leq 4 C C_g t_0^{1/2} | f |_{\infty}, 
 \end{align*}
 therefore $ t_0 \geq 1/{( 16 C^2 C_g^2 | f |_{\infty}^2)}={c_0}/{|f|_{\infty}^2} $ which is a contradiction. There (3.2) must hold. The estimate $ T_f > c_0/{| f |_{\infty}^2} $ is valid since $ \limsup_{t \to T_f} | u(t) |_{\infty} = \infty $ if $ T_f $ is finite.
 \end{proof}
 \bigskip
  Now we prove estimate (1.11)  of Theorem 1.2 by induction on $j$. Let $j \geq 1 $, and  assume 
 \begin{align}
 t^{k/2} | \mc D^k u (t) |_{\infty} \leq K_k | f|_{\infty}, \ts \text{for} \ts  0 \leq t \leq  \frac{c_0}{|f|_{\infty}^2}  \ts \text{and} \ts 0 \leq  k \leq j-1. 
 \end{align}
Let us apply  $ D^j$ to the equation $u_t = \triangle u + D_i \mb P g(u) $ to obtain 
 \begin{align*}
 v_t = \triangle v + D^{j+1} \mb Pg(u), \ts  v:= D^j u ,\\
 v(t) = D^j e^{\triangle t } f + \int_0^t e^{\triangle (t-s)} D^{j+1} (\mb P g(u))(s) ds .
 \end{align*}
 Using (2.4) we get 
 \begin{align}
 t^{j/2} |v(t)|_{\infty} \leq C |f|_{\infty} + t^{j/2} \biggl  |\int_0^t e^{\triangle (t-s)} D^{j+1} ( \mb P g (u))(s)   ds \biggr |_{\infty}.  
 \end{align}
We split the integral into 
 \begin{align*}
 \int_0^{t/2} + \int_{t/2}^t =: I_1 + I_2 
 \end{align*}
 and obtain 
\begin{align*}
 |I_1(t)|  & = \biggl | \int_0^{t/2} D^{j+1} e^{\triangle (t-s) } ( \mb P g(u))(s) ds \biggr |_{\infty} \\
 & \leq \int_0^{t/2} |D^{j+1} e^{\triangle (t-s)} (\mb P g (u )) (s) ds |_{\infty} ds .
 \end{align*}
 Using the inequality (2.5)  in Lemma 2.1, we get 
 \begin{align*}
 | I_1(t) |_{\infty} & \leq C  \int_0^{t/2} (t-s)^{-(j+1)/2} |g(u(s))|_{\infty} ds \\
 & \leq  C |f |_{\infty}^2 t^{(1-j)/2}. \\
 \end{align*}
 The integrand in $I_2$ has singularity at $s=t$. Therefore, we can move only one derivative from $D^{j+1} \mb P  g(u)$ to the heat semigroup.( If we move two or more derivatives then the singularity becomes non-integrable.) Thus, we have 
 \begin{align*}
 |I_2(t)|_{\infty}   = \biggl | - \int_{t/2}^t De^{\triangle (t-s)} (D^j  \mb P g(u))(s)  ds \biggr | _{\infty}. \nonumber 
 \end{align*}
 Since the Leray projector commutes with any order  derivatives, therefore 
 \begin{align*}
 |I_2(t) |_{\infty}  = \bigg| - \int_{t/2}^t D e^{\triangle (t-s) } ( \mb P D^j g (u))(s) ds \bigg|_{\infty}.
 \end{align*}
 If we use Lemma 2.1 for $j=1$, we obtain 
 \begin{align}
 |I_2(t) |_{\infty} \leq C \int_{t/2}^t (t-s)^{-1/2} |D^j g(u)(s) |_{\infty} ds. 
 \end{align}
 Since $g(u)$ is quadratic in $u$, therefore
 \begin{align*}
 |D^j  g(u) |_{\infty} \leq C | u|_{\infty} | \mc D^j u |_{\infty} +  \sum_{k=1}^{j-1} | \mc D^k u |_{\infty} | \mc D^{j-k} u |_{\infty} .
 \end{align*}
 By induction hypothesis (3.3) we obtain 
 \begin{align}
 \sum_{k=1}^{j-1} | \mc D^k u (s)  |_{\infty} | \mc D^{j-k} u (s)  |_{\infty} \leq C s^{-j/2} | f |_{\infty}^2. 
 \end{align}
 Expression in (3.5) can be estimated as below: 
 \begin{align*}
  |I_2(t) |_{\infty}  &\leq C \int_{t/2}^t (t-s)^{-1/2} \bigg(  C | u(s)|_{\infty} | \mc D^j u(s)  |_{\infty} +  \sum_{k=1}^{j-1} | \mc D^k u (s)  |_{\infty} | \mc D^{j-k} u (s)  |_{\infty}  \bigg) ds \\
&= J_1 + J_2. 
 \end{align*}
Using (3.6), and since  $ \int_{t/2}^t (t-s)^{-1/2} s^{-j/2} ds = C t^{(1-j)/2}$, where $C$ is independent of $t$, we obtain $|J_2(t) |_{\infty} \leq C | f |_{\infty}^2 t^{(1-j)/2}$.
 \medskip \\ 
 For $J_1$, we have 
 \begin{align*}
 | J_1(t) |_{\infty}  &= C \int_{t/2}^t (t-s)^{-1/2} |u(s)|_{\infty} | \mc D^j u(s) |_{\infty} ds  \\
 & \leq C |f|_{\infty} \int_{t/2}^t (t-s)^{-1/2} s^{-j/2} s^{j/2} |\mc D^j u (s) |_{\infty} ds \\
 & \leq C |f|_{\infty} t^{(1-j)/2} \mathop{\max}_{0 \leq s \leq t } \{ s^{j/2} \mc D^j u(s) |_{\infty} \}. 
 \end{align*}
 We use these bounds to bound the integral in (3.4). We have $v= D^j u $. Then maximizing the resulting estimate for $t^{j/2} |D^j u(t)|_{\infty}$ over all derivatives $D^j$ of order $j$ and setting 
\begin{align*}
\phi (t) := t^{j/2} |\mc D^j u(t)|_{\infty} 
\end{align*}
and from (3.4), we obtain the following estimate 
\begin{align*}
\phi(t) \leq C |f|_{\infty} +C t^{1/2} |f|_{\infty}^2 +C |f|_{\infty} t^{1/2} \mathop{\max}_{0\leq s \leq t} \phi (s) \ts \text{for} \ts 0 \leq t \leq \frac{c_0}{|f|_{\infty}^2 }.
\end{align*}
Since $t^{1/2} |f|_{\infty} \leq \sqrt{c_0} $ then $C t^{1/2}  |f|_{\infty}^2 \leq C \sqrt{c_0} |f|_{\infty} $. Therefore
\medskip
\begin{align}
\phi (t) \leq C_j |f|_{\infty} + C_j |f|_{\infty} t^{1/2} \mathop{\max}_{0\leq s \leq t} \phi (s)  \ts \text{for} \ts 0 \leq   t \leq c_0/{|f|_{\infty}^2 }.
\end{align}
Let us fix $C_j$ so that the above estimate holds, and set 
\begin{align*}
c_j = \min \bigg\{ c_0 , \frac{1}{4 C_j^2} \bigg\}.
\end{align*}
\medskip
First, let us prove the following
\begin{align*}
\phi(t) < 2 C_j |f|_{\infty}  \ts \text{for} \ts 0 \leq  t <  \frac{c_j }{|f|_{\infty}^2}.
\end{align*}
Suppose there is a smallest time  $t_0$ such that $ 0 < t_0 < c_j/{|f|_{\infty}^2 }  $ with $\phi(t_0) = 2C_j |f|_{\infty} $. Then using (3.7) we obtain
\begin{align*}
2 C_j |f|_{\infty} = \phi (t_0) \leq C_j |f|_{\infty} + 2 C_j^2  |f|_{\infty}^2 t_0^{1/2} ,
\end{align*}
thus 
\begin{align*}
1\leq 2 C_j |f|_{\infty} t_0^{1/2}  \ts \text{gives} \ts t_0 \geq c_j/{|f|_{\infty}^2}  
\end{align*}
which contradicts the assertion. Therefore, we proved the estimate 
\begin{align}
t^{j/2} |\mc D^j u(t) | _{\infty} \leq 2 C_j |f|_{\infty}   \ts \text{for} \ts 0 \leq  t \leq  c_j/{|f|_{\infty}^2 }.
\end{align}
If 
\begin{align}
T_j:= \frac{c_j}{|f|_{\infty}^2} < t \leq \frac{c_0}{|f|_{\infty}^2}=: T_0
\end{align}
then we start the corresponding estimate at $t-T_j$. Using Lemma 3.1, we have $|u(t-T_j)|_{\infty} \leq 2 |f|_{\infty}$ and obtain 
\begin{align}
T_j^{j/2} |\mc D^j u(t) |_{\infty} \leq 4 C_j |f|_{\infty} . 
\end{align}
Finally, for any $t$ satisfying (3.9) 
\begin{align*}
t^{j/2} \leq T_0^{j/2} = \bigg( \frac{c_0}{c_j} \bigg)^{j/2} T_j^{j/2}
\end{align*}
and (3.10) yield 
\begin{align*}
t^{j/2} |\mc D^j u(t) |_{\infty} \leq 4 C_j  \bigg( \frac{c_0}{c_j} \bigg)^{j/2} |f|_{\infty}.
\end{align*}
This completes the proof of Theorem 1.2.

\section{Estimates For the Navier-Stokes Equations}
\bigskip

Recall the transformed abstract ordinary differential equation (1.6)
\begin{align}
u_t = \triangle u - \mb P (u \cdot \nabla  u), \ts \nabla \cdot u = 0 
\end{align}
with 
\begin{align}
   u(x,0)=f(x).
\end{align}
Solution of (4.1) and (4.2) is given by 
\begin{align}
u(t) = e^{\triangle t } f -\int_0^t  e^{ \triangle (t-s) } \mb P(u \cdot \nabla u)(s)  ds .
\end{align}
\medskip 

Using (4.3) with previous  estimates  (2.3), (2.4) and (2.5), we prove the following lemma.
\bigskip

\begin{Lemma}
Set 
 \begin{align}
 V(t)= |u(t)|_{\infty} + t^{1/2} |\mc D u(t) |_{\infty} , \ts 0 <  t < T(f). 
 \end{align}

 There is a constant $C>0$, independent of $t$ and $f$, so that 
 \begin{align}
 V(t) \leq C|f|_{\infty} + C t^{1/2} \mathop{\max}_{0 \leq s \leq t }{V^2(s)}  , \ts 0 < t < T(f). 
 \end{align}
 \end{Lemma}
  \begin{proof} Using estimate (2.3)  of the heat equation  in (4.3), we obtain
 \begin{align*}
 |u(t)|_{\infty}  & \leq |f|_{\infty} + \bigg| \int_0^t e^{\triangle(t-s) }  \mb P (u \cdot \nabla u )(s) ds \bigg|_{\infty}. 
 \end{align*}
 Apply identity $\mb P (u\cdot \nabla u ) = \sum_i D_i \mb P (u_i u )  $ with  the fact,  heat semi-group commutes with $D_i$, then use of inequality (2.5) in  Lemma 2.1 for $j=1$ to proceed
\begin{align*}
 | u ( t)|_{\infty} & \leq  |f|_{\infty}  + C \int_0^t (t-s)^{-1/2} |u(s)|_{\infty}^2 ds \\
 &  =  |f|_{\infty} + C \int_0^t (t-s)^{-1/2}  s^{-1/2} s^{1/2} |u(s)|_{\infty}^2 ds \\
& \leq |f|_{\infty} + C \mathop{\max}_{0\leq s \leq t } \{ s^{1/2} |u(s)|_{\infty}^2  \} \int_0^t (t-s)^{-1/2}  s^{-1/2} ds.
\end{align*}
Since $  \int_0^t (t-s)^{-1/2}  s^{-1/2} ds = C>0 $, which is independent of $t$,  we have the following estimate 
 \begin{align}
 |u(t)|_{\infty}  & \leq | f|_{\infty} +  C  \mathop{\max}_{0\leq s \leq t } \{ s^{1/2} |u(s)|_{\infty}^2 \}  \nonumber \\
 | u(t)|_{\infty}  & \leq |f|_{\infty} + C t^{1/2}  \mathop{\max}_{0\leq s \leq t }{V^2(s) }. 
 \end{align}
 Apply $D_i$ to (4.1), and the {\it Duhamel's principle}  to obtain
 \begin{align}
 v(t) = D_i e^{\triangle t} f - \int_0^t e^{ \triangle (t-s)  } D_i \mb P (u \cdot \nabla ) u(s) ds .
 \end{align}
We can estimate the integral in (4.7) using Lemma 2.1 for $j=1$ in the following way: 
\begin{align*}
  \bigg|\int_0^t D_i e^{ \triangle (t-s)}  \mb P (u \cdot \nabla  u)(s) ds \bigg| &  \leq  \int_0^t |D_i e^{ \triangle  (t-s)}  \mb P(u\cdot \nabla u) (s) | ds \\
 & \leq C \int_0^t (t-s)^{-1/2} |u(s)|_{\infty} | \mc D u(s)|_{\infty} ds \\
 &= C \int_0^t (t-s)^{-1/2} s^{-1/2} s^{1/2} |u(s)|_{\infty} |\mc D u(s) |_{\infty} ds \\
 & \leq C \mathop{\max}_{0 \leq s \leq t } { \{s^{1/2} | u(s) |_{\infty} |\mc D u(s)|_{\infty}\}  }   \int_0^t (t-s)^{-1/2}  s^{-1/2} ds \\
 & \leq C \mathop{\max}_{0\leq s \leq t } { \{ |u(s)|_{\infty}^2 + s | \mc D u(s) |_{\infty}^2 \} } .
\end{align*}
Therefore,  using (2.4)  $j=1$ in expression  (4.7), we arrive at 
\begin{align}
|v(t) |_{\infty}  & \leq C t^{-1/2} |f|_{\infty} + C  \mathop{\max}_{0\leq s \leq t } { \{ |u(s)|_{\infty}^2 + s | \mc D u(s) |_{\infty}^2 \} }  \nonumber  \\
t^{1/2} |\mc D u(t) |_{\infty}  & \leq C |f|_{\infty} + C t^{1/2}  \mathop{\max}_{0\leq s \leq t } {V^2(t) }. 
\end{align}
 Using  (4.6) and (4.8), we have proved Lemma 4.1. 
\end{proof}
\bigskip

\begin{Lemma}
Let $C>0$  denote the constant in estimate (4.5) and set 
\begin{align*}
c_0 = \frac{1}{16C^4}.
\end{align*}
Then $T_f  > c_0 /{|f|_{\infty}^2}$  and 
\begin{align}
| u(t) |_{\infty} + t^{1/2} | \mc D u (t) |_{\infty} < 2C |f|_{\infty} \ts  \text{for}  \ts 0 \leq t < \frac{c_0}{|f|_{\infty}^2}. 
\end{align}
\end{Lemma}
\begin{proof}
We prove this lemma by contradiction after recalling the definition of $V(t)$ in (4.4). Suppose that (4.9) does not hold, then denote by $t_0$ the smallest time with $V(t_0)= 2C |f|_{\infty}$. Use (4.5)  to obtain 
\begin{align*}
2C| f |_{\infty}  & = V(t_0) \\
&  \leq C |f|_{\infty} + C t_0^{1/2} 4 C^2 | f |_{\infty}^2, 
\end{align*}
thus 
\begin{align*}
1 \leq 4 C^2 t_0^{1/2} | f |_{\infty}^2 ,
\end{align*}
therefore $t_0 \geq c_0/ | f |_{\infty}^2$. This contradiction proves (4.9) and $T_f > c_0/{|f |_{\infty}^2}$.
\end{proof}
\medskip 
Lemma 4.2 proves Theorem 1.1 for $j=0$ and $j=1$. By an induction argument as in the proof of Theorem 1.2 one proves Theorem 1.1 for any $j =0,1,\cdots $

\section{Remarks}
\bigskip

We can apply estimate (1.5) of Theorem 1.1  for 
\begin{align}
\frac{c_0}{2 | f |_{\infty}^2} \leq t \leq \frac{c_0}{ | f |_{\infty}^2}
\end{align}
and obtain 
\begin{align}
| \mc D^j u (t) |_{\infty} \leq C_j | f |_{\infty}^{j+1} 
\end{align}
 in interval (5.1). Starting the estimate at $ t_0 \in [0, T_f) $ we have 
 \begin{align}
 | \mc D^j u ( t_0 + t ) |_{\infty} \leq C_j | u(t_0) |_{\infty}^{j+1} 
 \end{align}
 for 
 \begin{align}
 \frac{c_0}{2 | u(t_0)|_{\infty}^2} \leq t \leq \frac{c_0}{ | u(t_0)|_{\infty}^2}.
 \end{align}
Then, if $t_1$ is fixed with 
\begin{align}
\frac{c_0}{2 | f|_{\infty}^2} \leq t_1 < T_f,
\end{align}
we can maximize both sides of (5.3) over $ 0 \leq t_0 \leq t_1$ and obtain
\begin{align}
\max \bigg\{ | \mc D^j u(t) |_{\infty } : \frac{c_0}{2 | f|_{\infty}^2} \leq t \leq t_1+ \tau \bigg\} \leq C_j \max \{ | u(t) |_{\infty}^{j+1} : 0 \leq t \leq t_1 \} 
\end{align}
with 
\begin{align*}
\tau = \frac{c_0}{| u (t_1) |_{\infty}^2} 
\end{align*}
\medskip \\
Estimate (5.6) says, essentially, that the maximum of the $j$-th derivatives of $u$ measured by $ | \mc D^j u|_{\infty} $ , can be bounded in terms of $ | u |_{\infty}^{j+1} $.
The positive value of $\tau$ on the left-hand side of (5.6) shows that $ |u|_{\infty}^{j+1}$ controls $ | \mc D^j u |_{\infty}$ for some time into the future. 
\medskip  \\
As is well  known, if $(u, p)$ solves the Navier-Stokes equations and $ \lambda >0$ is any scaling parameter, then the functions $ u_{\lambda}, p_{\lambda} $ defined by 
\begin{align*}
u_{\lambda} (x,t) = \lambda u( \lambda x, \lambda^2 t) , \ts  p_{\lambda} ( x,t) = \lambda^2 p( \lambda x , \lambda^2 t ) 
\end{align*}
also solve the Navier-Stokes equations. Clearly,
\begin{align*}
| u_{\lambda} (t) |_{\infty} = \lambda | u ( \lambda^2 t)|_{\infty} , \ts | \mc D^j u_{\lambda} (t) |_{\infty} = \lambda^{j+1} | \mc D^j u ( \lambda^2 t ) |_{\infty}. 
\end{align*}
Therefore, $ | \mc D^j u |_{\infty} $ and $ | u|_{\infty}^{j+1} $ both scale like $ \lambda^{j+1} $, which is, of course, consistent with the estimate (5.6). We do not know under what assumptions $ | u |_{\infty}^{j+1} $ can conversely be estimated in terms of $ | \mc D^j u |_{\infty} $.
		
\end{document}